\documentclass[12pt,a4paper]{article}

\usepackage{url,epsfig,amssymb,amsthm,amsmath}
\usepackage[active]{srcltx}
\usepackage[utf8]{inputenc}
\usepackage[a4paper,hcentering,vcentering]{geometry}
\graphicspath{{Fig/}}

\usepackage{hyperref}

\newtheorem{corollary}{Corollary}
\newtheorem{lemma}{Lemma}
\newtheorem{proposition}{Proposition}

\newcommand{\R}{\ensuremath{\mathbb{R}}}
\newcommand{\T}{\ensuremath{\mathbb{T}}}

\newcommand{\eps}{\varepsilon}
\newcommand{\de}{\delta}

\title{A remark on the onset of resonance overlap}%
\author{Jacques Fejoz\thanks{CEREMADE/Universit{\'e} Paris-Dauphine \& IMCCE/Observatoire de Paris, CNRS, PSL Research University, Paris, France, \url{jacques.fejoz@psl.eu}}, \enspace Marcel Guardia\thanks{Departament de Matem\`atiques i Inform\`atica, Universitat de Barcelona, Gran Via, 585, 08007 Barcelona, Spain and Centre de Recerca Matem\`atica, Edifici C, Campus Bellaterra, 08193 Bellaterra, Spain, \url{guardia@ub.edu}}}

\begin{document}

\maketitle

\begin{abstract}
    Chirikov's celebrated criterion of resonance overlap has been widely used in celestial mechanics and Hamiltonian dynamics to detect global instability, but is rarely rigourous. We introduce two simple Hamiltonian systems, each depending on two parameters measuring respectively the distance to resonance overlap and non-integrability. Within some thin region of the parameter plane, classical perturbation theory shows the existence of global instability and symbolic dynamics, thus illustrating Chirikov's criterion.
\end{abstract}



\section{Heuristic introduction -- resonance overlapping}

Let $H : \T \times \R \times \T \to \R$ be a time-dependent Hamiltonian of class
$\mathcal{C}^\infty$, $2\pi$-periodic in time, of the form
\begin{equation}\label{def:H0:00}
  H(x,y,t) = H_0(y) + \eps F(x,y,t).
\end{equation}
When $\eps=0$, the time-$2\pi$ map $\phi$ of the flow of $H$ is integrable and the level curves of the coordinate $y$ are all invariant. Curves whose rotation number $H_0'(y)$ is rational or have good rational approximations disappear for generic Fourier coefficients of $F$, as Poincar{\'e} noticed~\cite{Poincare:1892}. In place of some of those resonant curves, periodic orbits originate, usually by elliptic/hyperbolic pairs.  (More generally, non-smooth invariant graphs known as Aubry-Mather sets can be found generically as the support of minimizing measures.) For systems of one and a half degree of freedom, like \eqref{def:H0:00} (or two degrees of freedom), elliptic orbits are surrounded by elliptic "eyes" (see \cite{Arnold:1989,Lichtenberg:1992, Meiss:2007} and references therein), where some kind of stability prevails over long time intervals (see \cite{Bounemoura:2015}). Simultaneously, as KAM theory proves, a positive Lebesgue measure of Diophantine invariant curves persist~\cite{Moser:1966a, Moser:1966}. As $\eps$ increases, more invariant curves disappear. (Some new invariant curves also show up, although these ones are harder to detect.) Persisting invariant curves are obstructions to large deviations in the $y$ direction. Note that, in higher dimension, Lagrangian invariant tori do not separate the phase space anymore and allow the dynamics for slow, non-local instability, called \emph{Arnold diffusion}~\cite{Arnold:1964, Chirikov:1979}. Arnold diffusion is notoriously difficult to show.

\bigskip As $\eps$ keeps increasing, the seeming sizes of resonant eyes grow. As long as the separatrices of two resonances are well apart, invariant curves separating the two zones confine orbits on one side or the other, and the system behaves as if the two resonances did not interact: in each zone, the dynamics is reasonably described by an integrable approximation retaining only the harmonics responsible for the opening of the relevant eye. Chirikov has conjectured that an orbit will start moving between two resonance eyes in a chaotic and unpredictable manner ``as soon as these unperturbed resonances overlap''~\cite{Chirikov:1960, Chirikov:1979}. For a modern reference, with applications to celestial mechanics, see Morbidelli's book~\cite[Chap.~6 and Section~9.2 in particular]{Morbidelli:2002}. Indeed, as soon as the separatrices of the two resonances get close to each other, the dynamics is no more described by two adjacent one-resonance integrable models and, as Morbidelli puts it, ``an initial condition in the overlapping region does not know which resonance it belongs to, and hesitates about which guiding trajectory it should follow". The criterion has been used for magnetically confined plasmas (as in Chirikov's initial work or Escande's review \cite{Escande:2016}), the Solar System (e.g. \cite{Morbidelli:1996, Morbidelli:2002, Nesvorny:1998, Petit:2017}), space debris \cite{Celletti:2017}, transport and turbulence in fluid mechanics~\cite{Castillo:1993}, as well as particle dynamics in accelerators, microwave ionization of Rydberg atoms, etc. (see~\cite{Lichtenberg:1992} and references therein). 

Defining the closeness of two resonance eyes, or their overlap, is not a simple matter, since generically separatrices split and thus do not precisely circumscribe an "eye". Physicists speak of a "stochastic layer" at the border, but little is really known about dynamics in this layer, apart from the horseshoe (a set of zero-measure) given by the Birkhoff-Smale theorem~\cite{Moser:1973}.
Moreover, there is a whole web of resonances, and, for each resonance, there are infinitely many ways to choose integrable approximations describing the opening of the corresponding eye. All this makes Chirikov's criterion imprecise. For a further analysis of why Chirikov's criterion fails in general, see for example~\cite{Benest:1995,Chandre:2002,Meiss:2007}.

\bigskip Key to instability is the destruction of invariant curves. The precise mechanism remains mysterious, despite extensive efforts (e.g. \cite{Forni:1996, Mather:1988}). One  attempt to describe whether invariant curves persist or not, which has been quite successful for practical purposes, is Greene's criterion, which analyzes the stability of accumulating periodic orbits. This criterion has been partly justified~\cite{Delshams:2000, Falcolini:1992, MacKay:1992}. Renormalization should also be an important tool for the full picture~\cite{Chandre:2002}. 

\bigskip We will not address this difficult issue directly, but rather aim at illustrating Chirikov's criterion on a simple, rigorous example. Consider the Hamiltonian 
\begin{equation}
  \label{def:hamiltonian1}%
  h_{\eps,\mu}(x,y,t)= \frac{y^2}{2}-\frac{y^3}{3}+\eps F(x,y,t;\mu); \quad F(x,y,t;\mu)=-\frac{1}{12}\cos x + \mu f(x,y,t)
\end{equation}
on $\T \times \R\times\T$, where $\eps$ and $\mu$ are real parameters and  $f(x,y,t)$ is a $\mathcal C^\infty$ time-periodic perturbation.

We denote by
\begin{equation}\label{def:H0}
  h_\eps(x,y) = \frac{y^2}{2} - \frac{y^3}{3} - \frac{\eps}{12} \cos x 
\end{equation}
the Hamiltonian $h_{\eps,\mu}$ with $\mu=0$.

The use of two parameters is reminiscent of Arnold's approach to Arnold Diffusion in his seminal paper \cite{Arnold:1964} where he separates the average with respect to all angles but $x$ (which he takes bigger, of size $\eps$) from the other harmonics (of size $\eps\mu$).

For $\mu=0$, the Hamiltonian \eqref{def:H0} can be  viewed as a modification of the twist Hamiltonian $y^2/2$ in the class of ``classical'' Hamiltonians (sums of a kinetic part depending only on $y$ and a potential part depending only on $x$ and $t$) involving only the lowest degree term in $y$ and lowest order harmonic in the average $\cos x$ creating two resonance eyes close to each other. In this sense, $h_{\eps,\mu}$ is a simple but somewhat general model family eligible to Chirikov's criterion of resonance overlapping (although in the parameter space or in the space of series coefficients the regime we consider is very specific).

As a variant we  also consider the doubly periodic Hamiltonian 
\begin{equation}
  \label{def:hamiltonian2}%
  \tilde h_{\eps,\mu}(x,y,t)= \cos y+ \eps F(x,y,t;\mu); \quad F(x,y,t;\mu)=- \cos x + \mu f(x,y,t)
\end{equation}
on $\T^3$, for which the instability is similar locally, but also more global due to the double periodicity. As before, we denote its first order by
\begin{equation}
  \label{def:H0:per}%
    \tilde h_\eps(x,y) = \cos y-\eps \cos x.
\end{equation}

Interestingly, the Hamiltonian $h_0$ (cubic in the actions) has a twistless curve (where the unperturbed frequency map $y \mapsto y(1-y)$ has a fold singularity). Greene's criterion has been applied to this twistless curve in~\cite{Castillo:1996}. Hamiltonians similar to $  \tilde h_\eps$ have been studied notably by Zaslavsky in~\cite{Zaslavsky:1991}, as examples displaying ``stochastic webs'' with spatial patterns. The Hamiltonian $  \tilde h_\eps$ is also a minimalistic Hamiltonian subcase of the so-called Arnold-Beltrami-Childress (ABC) flow, for which we refer to~\cite{Zhao:1993}. This later work, to which the present paper is closely related, uses Melnikov theory for proving the existence of chaotic solutions.

\begin{figure}[t]
  \centering
  \includegraphics[scale=1]{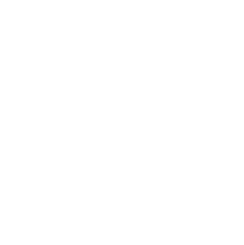}
  \includegraphics[scale=1]{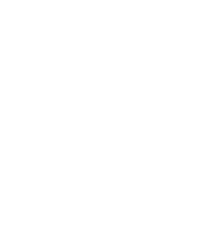}
  \includegraphics[scale=1]{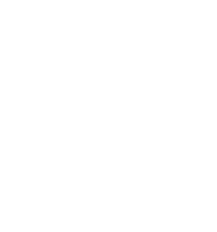}
  \caption{Level curves of $h_\eps$ for $\eps=\frac{1}{2}$, $1$, $\frac{3}{2}$}
  \label{fig:h0}
\end{figure}

\begin{figure}[t]
  \centering
  \includegraphics[scale=1]{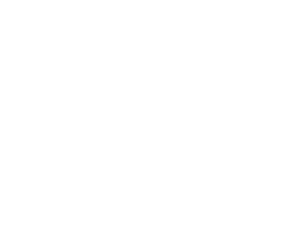}
  \includegraphics[scale=1]{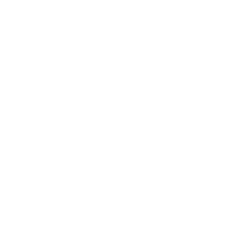}
  \includegraphics[scale=1]{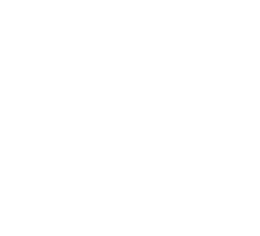}
  \caption{Level curves of $  \tilde h_\eps$ for $\eps=\frac{1}{2}$, $1$, $\frac{3}{2}$} 
  \label{fig:h0'}
\end{figure}

\section{An illustration of the overlapping principle}

We will now describe the Hamiltonians \eqref{def:hamiltonian1} and \eqref{def:hamiltonian2}, where one can quite explicitly see the transition from complete integrability to resonance overlapping. The key point in these examples is that the two parameters controlling the appearance of (and distance between) resonance eyes and the non-integrability are decoupled. 

The same overlapping of resonances behavior would take place if $\cos x$ in \eqref{def:H0} and \eqref{def:H0:per} would be replaced by any potential $V(x)$ with unique non-degenerate maximum and minimum, which is a generic condition. We state the results just for the models \eqref{def:H0} and \eqref{def:H0:per} for the sake of simplicity.

The phase portraits of Hamiltonians $h_\eps$ and $  \tilde h_\eps$ defined above are shown in Figs. \ref{fig:h0} and \ref{fig:h0'} for different values of the parameter $\eps$. The interesting bifurcation value for us will be $\eps=1$, so henceforth we will assume that $\eps>0$. Both Hamiltonians \eqref{def:H0} and \eqref{def:H0:per} are integrable but do not have global action-angle coordinates (as they would for $\eps=0$) --namely, they have separatrices which create ``eyes'' of width  $O(\sqrt{\eps})$ in the $(x,y)$-coordinates. Also, they possess hyperbolic critical points at 
\begin{equation}
\label{def:saddles}
\begin{cases}
 ((2k+1)\pi, 0)\quad\text{ and }\quad (2k\pi, 1) \qquad &\text{for }h_\eps\\
  (2k\pi, 2k'\pi)\quad\text{ and }\quad ((2k+1)\pi, (2k'+1)\pi) \qquad &\text{for }  \tilde h_\eps
\end{cases}
\end{equation}
for every $k,k'\in\mathbb{Z}$. The two Hamiltonians undergo a bifurcation (sometimes called a \emph{heteroclinic reconnection}) when the energy levels of the two families of hyperbolic points coincide, namely for $\eps=1$.

For $0<\eps<1$, the separatrices attached to the hyperbolic critical points are graphs over the $x$ direction. This implies that there are invariant curves separating the saddles having different $y$ component.  At the bifurcation value, the net of separatrices changes its topology by creating heteroclinic connections between  saddles with different $y$ component. In particular, all smooth invariant curves separating the two resonance eyes of  $h_\eps$ may break down, and the Hamiltonian $  \tilde h_\eps$ might not have any invariant smooth curve over the $x$- or $y$-axes.

Consider now the perturbed Hamiltonians \eqref{def:hamiltonian1} and \eqref{def:hamiltonian2} with $0<\mu\ll 1$.
Note that now the parameter $\eps$ measures the size of resonant eyes, while
$\mu$ measures the distance to the integrable approximations $h_\eps$ and $  \tilde h_\eps$. 
Note that those Hamiltonians are of the form \eqref{def:H0:00}.

The parameter $\varepsilon$ measures hyperbolicity (i.e. the size of eyes of resonance), while $\varepsilon\mu$ measures the distance to some integrability. The real, difficult case for Chirikov's criterion is when $\mu=1$, i.e. when non-integrable terms are of the same size as the integrable terms responsible for the eyes of resonance.  The goal of the present work is to illustrate Chirikov's criterion within a thin region in the parameter plane, defined by $ |\eps-1|\ll 1$ and $|\mu|\ll 1$.\footnote{It is a matter of definition whether our analysis in the regime where $|\mu| \ll 1$ may indeed be called ``Chirikov's criterion'': on the one hand, this regime is eligible to being analysed with classical, perturbative tools, as we show below and contrary to the standard domain of application of the criterion; on the other hand, it does correspond to the general idea of overlapping eyes of resonance.} Our main result is the following.

\begin{proposition}\label{prop:VerticalConnection}
  There exist constants $C_1, C_2, \mu_0>0$ such that for any
  $\mu\in (0,\mu_0)$,
  \begin{enumerate}
  \item for $0<\eps<C_1<1$, heteroclinic connections between the periodic 
orbits $\mu$-close to the  saddles   \eqref{def:saddles} with different 
$y$-component is not possible.
  \item for $1-C_2\mu<\eps\leq 2$, $h_{\eps,\mu}$ possesses transversal heteroclinic
    connections between the same periodic orbits as there are for the
     case $\eps=1$ and $\mu=0$.
  \end{enumerate}
\end{proposition}

This result can be seen as the process of overlapping resonances in a non-integrable Hamiltonian system. In Regime 1 KAM curves prevent overlapping between the considered resonances: dynamics is confined between the invariant curves. In contrast, in Regime 2 all KAM curves break down and there is overlapping between resonances.

Note that here we are only considering the strongest resonances, corresponding to $\dot x=0$.  Certainly, the Hamiltonians $h_{\eps,\mu}$ and $  \tilde h_{\eps,\mu}$ possess many more at $\dot x\in\mathbb{Q}$ but they are much weaker (so we would need $\mu$ much smaller in order to apply the same arguments).

\bigskip We will now prove the proposition. For $\mu$ small enough, $h_{\eps,\mu}$ and $\tilde h_{\eps,\mu}$ have hyperbolic periodic orbits $\mu$-close to the saddles of $h_{\eps}$ and $ \tilde h_{\eps}$ respectively for any $\eps\in(0,2]$. Melnikov Theory \cite{Melnikov63} implies that the separatrices of $h_{\eps}$ and $ \tilde h_{\eps}$ usually break down.

We will use the following two lemmas. 

\begin{lemma}\label{lemma:Melnikov1}
  Fix $\eps>0$. For a generic $f$ there exists $\mu_0>0$ such that for all $\mu\in (0,\mu_0)$ all the separatrices of the Hamiltonians $h_\eps$ and $  \tilde h_\eps$ break down and the resulting invariant manifolds intersect transversally.
\end{lemma}

In this statement, \textit{generic} actually means that $f$ may be chosen in an open and dense subset of the functional space (typically the space of analytic Hamiltonians, or of Hamiltonians of classe $C^r$ with $r$ large enough). See \cite{Chen:2022} for a (much more advanced) discussion in this direction.

Note that this result is significantly different for $\eps\neq 1$ and $\eps=1$ since the separatrices for $h_\eps$ are different in both cases. Moreover, since we are interested in $\eps$ close to 1 and depending on $\mu$, one can also prove the following more precise lemma, which is also a consequence of Melnikov Theory (\textit{ibid.}).

\begin{lemma}\label{lemma:Melnikov2}
  There exist $\de_0>0$ and $\mu_0>0$ small such that for a generic $f$ with $\|f\|_{\mathcal C^2}\leq 1$, for all $\mu\in (0,\mu_0)$ and $\eps\in [1-\de_0,1+\de_0]$ all the separatrices of the Hamiltonians $h_\eps$ and $  \tilde h_\eps$ break down and the corresponding invariant manifolds intersect transversally.
\end{lemma}

\begin{proof}
We prove the lemma for $h_{\eps,\mu}$ but the same proof applies to $  \tilde h_{\eps,\mu}$. Writing $\eps=1+\de$, $h_{\eps,\mu}$ becomes
\[
h_{\eps,\mu}(x,y,t)= \frac{y^2}{2}-\frac{y^3}{3}-\cos x + F(x,y,t)
\]
where
\[
F(x,y,t)=-\de \cos x+\mu(1+\de) f(x,y,t)
\]
satisfies $F\sim O(\de,\mu)$. Then, for $\de$ and $\mu$ small enough one can apply Melnikov Theory to prove that, for a generic $f$, the heteroclinic connections of the integrable Hamiltonian break down and that the stable and unstable invariant manifolds of the corresponding saddles intersect transversally. Note that Melnikov Theory only gives the transversality of the invariant manifolds which coincided for $\frac{y^2}{2}-\frac{y^3}{3}-\cos x$, that is invariant manifolds of different saddles. Nevertheless, the Lambda lemma \cite{PalisM82} implies that, for $\de$ and $\mu$ small enough, the stable and unstable invariant manifold of one saddle also intersect transversally giving rise to transversal homoclinic connections.
\end{proof}

A consequence of these lemmas is that, for $0<\eps<1+\de$ (dependent or independent of $\mu$) all homoclinic separatrix connections associated with the periodic orbits of $h_{\eps,\mu}$ and $  \tilde h_{\eps,\mu}$ split. 
Moreover, they imply (together with KAM Theory) the existence of two regimes: separation of resonances and overlapping, which are given by the existence (or absence) of heteroclinic connections between saddles with different $y$-component.

\begin{proof}[End of proof of proposition~\ref{prop:VerticalConnection}]
  Statement 1 is a direct consequence of KAM Theorem. Indeed, the
  Hamiltonians $h_{\eps}$ and $  \tilde h_{\eps}$ have invariant curves which
  ``separate the saddles'' with different $y$-component and are graphs
  over the base $y=0$. Moreover, except at one curve, $h_{\eps}$ and
  $  \tilde h_{\eps}$ are non-degenerate at those curves (the period of the
  periodic orbit is changing with its energy). Therefore, one can
  apply KAM Theorem for time-periodic Hamiltonians to show existence
  in the extended phase space of a positive measure set of 2
  dimensional invariant tori close to those periodic orbits of
  $h_{\eps}$ and $  \tilde h_{\eps}$.

 Statement 2 of this proposition is a direct consequence of Lemma \ref{lemma:Melnikov2}.
\end{proof}

Proposition \ref{prop:VerticalConnection} has several consequences. Constant $C_2>0$ refers to Proposition~\ref{prop:VerticalConnection}. 

\begin{corollary}
  For $\mu>0$ small enough and $\eps>1-C_2\mu$, both $h_{\eps,\mu}$ and $  \tilde h_{\eps,\mu}$ have a compact invariant subset carrying symbolic dynamics with random excursions in the $y$ direction, of amplitude uniform with respect to both $\mu$ and $\eps$.
\end{corollary}

Indeed, consider for example the four saddle points $(\pm \pi,\pm \pi)$ of $  \tilde h_\eps$. In the neighborhood of these equilibria, a classical construction leads to a subshift of four symbols by using the concatenation of heteroclinic connections through $(0,0)$, $(\pm 2\pi,0)$ and $(0,\pm2\pi)$. This subshift is not a full shift, since for instance one cannot go directly from (neighborhoods of) $(\pi,\pi)$ to $(-\pi,-\pi)$ without passing through neighborhoods of either the other two, but these obvious obstructions are the only obstructions. One actually gets subshifts of arbitrarily many symbols by considering neighborhoods of correspondingly many saddles. We refer for any classical book in Dynamical Systems for this construction, and for example to \cite{Brin:2002}.

In Regime 1 of Proposition \ref{prop:VerticalConnection}, one also certainly has symbolic dynamics but it is confined in the vertical direction by the KAM curves. 

The Lambda lemma \cite{PalisM82} implies the following for $  \tilde h_{\eps,\mu}$. 

\begin{corollary}
  Let $y_+>y_-$, $\mu>0$ small enough and $\eps>1-C_2\mu$. The Hamiltonian $  \tilde h_{\eps,\mu}$ has orbits which travel from $y=y_-$ to $y=y_+$. Moreover, one can achieve such transition in time $T\sim (y_+-y_-) \, |\ln\mu|$.
\end{corollary}

This behavior is not possible in Regime 1 due to KAM curves. Note that this behavior is still not possible for the Hamiltonian  $h_{\eps,\mu}$ for $\mu>0$ small enough and $\eps>1-C_2\mu$ since it has invariant curves surrounding the two overlapped resonances. 

\section{Numerics}

So-described instabilities easily show numerically; see Fig.~\ref{fig:1}. Despite the exponential divergence of solutions, approximate computations in the above two regimes are justified by the Lambda lemma, which entails that computed pseudo-orbits are shadowed by true
orbits of $  \tilde h_{\eps,\mu}$.

\begin{figure}[htbp]
  \centering
  \includegraphics[scale=0.3]{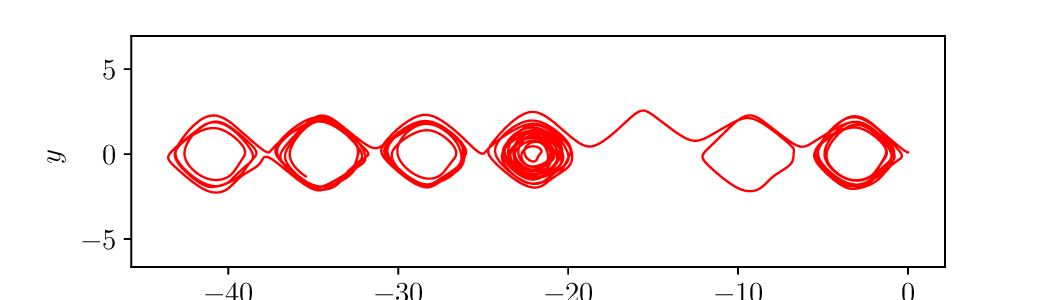}\\
  \includegraphics[scale=0.3]{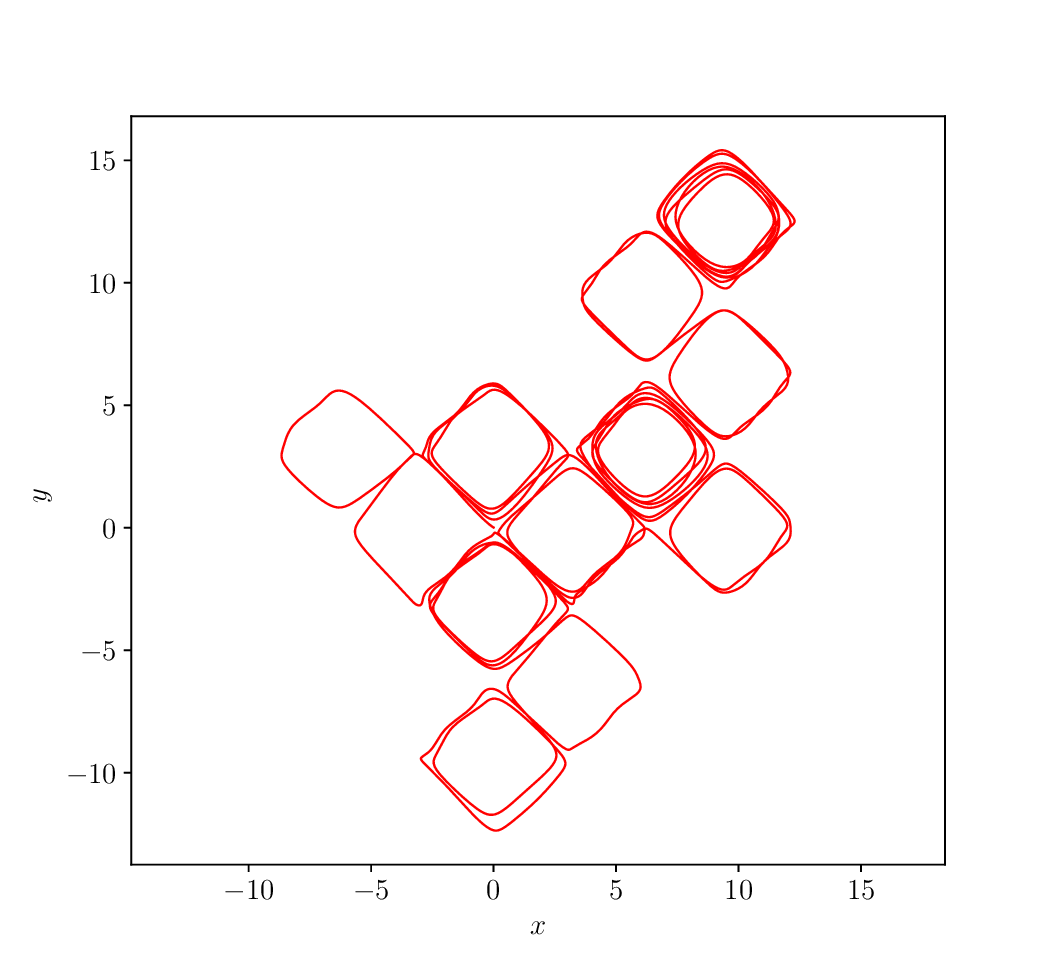} 
  \caption{Examples of unstable orbits of
    $h_{\eps,\mu}'(x,y,t) = \cos(y) - \eps\cos(x) + \mu \cos(x+2y+t)$, for $t \in [0,500]$.
    On the top, $\eps=0.8$, $\mu=0.1$, initial condition close to
    $(0,10^{-1})$. On the bottom, $\eps=0.99$, $\mu=0.1$, initial
    condition close to $(0,10^{-2})$.}
  \label{fig:1}
\end{figure}

\section*{Acknowledgements}

We have greatly benefitted from discussions with Cristel Chandre, Philip Morrison and Tere Seara, and from many suggestions of the referees. 

\section*{Funding}

M. Guardia has been supported by the European Research Council (ERC) under the European Union's Horizon 2020 research and innovation programme (grant agreement No. 757802). This work is part of the grant PID-2021-122954NB-100 funded by MCIN/AEI/10.13039/501100011033 and ``ERDF A way of making Europe''. M. Guardia is also supported by the Catalan Institution for Research and Advanced Studies via an ICREA Academia Prize 2019. This work is also supported by the Spanish State Research Agency, through the Severo Ochoa and María de Maeztu Program for Centers and Units of Excellence in R\&D (CEX2020-001084-M). This work is also supported by the project of the French Agence Nationale pour la Recherche CoSyDy (ANR-CE40-0014).

\section*{Conflict of interest}

The authors declare that they have no conflicts of interest.

\bibliographystyle{plain}
\bibliography{references} 

\end{document}